\documentclass[11pt, english]{article}

\usepackage[margin= 2 cm]{geometry}

\usepackage{amsthm}
\usepackage{amsmath}
\usepackage{amssymb}
\usepackage{setspace}
\usepackage{mathtools}
\usepackage{verbatim}
\usepackage{csquotes}
\usepackage{graphicx}
\usepackage[hidelinks]{hyperref}
\usepackage{thm-restate}
\usepackage{cleveref}
\usepackage{graphicx}
\usepackage{appendix}
\usepackage[inline]{enumitem}
\usepackage{framed}
\usepackage{subcaption}

\usepackage{bbm}
\usepackage{mathrsfs}

\usepackage{caption}

\usepackage{floatrow}
\usepackage[T1]{fontenc}

\usepackage{tikz}
\usepackage{mathdots}
\usepackage{xcolor}
\usepackage{diagbox}
\usepackage{colortbl}
\usepackage[absolute,overlay]{textpos}

\graphicspath{ {./images/} }

\usetikzlibrary{calc}
\usetikzlibrary{decorations.pathreplacing}
\usetikzlibrary{positioning,patterns}
\usetikzlibrary{arrows,shapes,positioning}
\usetikzlibrary{decorations.markings}

\tikzstyle{edge}=[very thick]
\definecolor{bostonuniversityred}{rgb}{0.8, 0.0, 0.0}
\definecolor{arsenic}{rgb}{0.23, 0.27, 0.29}
\tikzstyle{diredge}=[postaction={decorate,decoration={markings,
        mark=at position .95 with {\arrow[scale = 1]{stealth};}}}]

\newcommand{\defPt}[3]{
    \def \pt {(#1, #2)}
    \coordinate [at = \pt, name = #3];
}

\tikzset{
   conn/.pic={
     \defPt{0.2}{-0.5}{q0}
     \defPt{-1}{-1.5}{q5}
    \defPt{1}{1.2}{q1}
    \defPt{1}{2.7}{q6}
    \defPt{1.25}{-1.2}{q2}
    \defPt{2.5}{0.6}{q3}
    \defPt{2.5}{-0.6}{q4}
  
        \draw[line width=1 pt] (q0) -- (q1) -- (q3) -- (q4);
        \draw[line width=1 pt] (q2) -- (q3);
        \draw[line width=1 pt] (q0) -- (q5);
        \draw[line width=1 pt] (q1) -- (q6);
  }
}

\newcommand{\fitellipsis}[3] 
{\draw []let \p1=(#1), \p2=(#2), \n1={atan2(\y2-\y1,\x2-\x1)}, \n2={veclen(\y2-\y1,\x2-\x1)}
    in ($ (\p1)!0.5!(\p2) $) ellipse [ x radius=\n2/2+0.3cm+#3cm, y radius=#3cm, rotate=\n1];
}

\newcommand{\fitellipsiss}[3] 
{\draw [fill=white]let \p1=(#1), \p2=(#2), \n1={atan2(\y2-\y1,\x2-\x1)}, \n2={veclen(\y2-\y1,\x2-\x1)}
    in ($ (\p1)!0.5!(\p2) $) ellipse [ x radius=\n2/2+#3cm, y radius=#3cm, rotate=\n1];
}

\newcommand{\fitellipsisss}[3] 
{\draw []let \p1=(#1), \p2=(#2), \n1={atan2(\y2-\y1,\x2-\x1)}, \n2={veclen(\y2-\y1,\x2-\x1)}
    in ($ (\p1)!0.5!(\p2) $) ellipse [ x radius=\n2/2+#3cm, y radius=#3cm, rotate=\n1];
}

\theoremstyle{plain}

\newtheorem*{thm*}{Theorem}
\newtheorem{thm}{Theorem}
\Crefname{thm}{Theorem}{Theorems}

\newtheorem*{lem*}{Lemma}
\newtheorem{lem}[thm]{Lemma}
\Crefname{lem}{Lemma}{Lemmas}

\newtheorem*{claim*}{Claim}

\Crefname{claim}{Claim}{Claims}
\Crefname{claim}{Claim}{Claims}

\newtheorem{prop}[thm]{Proposition}
\Crefname{prop}{Proposition}{Propositions}

\Crefname{cor}{Corollary}{Corollaries}

\Crefname{conj}{Conjecture}{Conjectures}

\Crefname{qn}{Question}{Questions}
\newtheorem*{qn*}{Question}

\Crefname{obs}{Observation}{Observations}

\newtheorem{ex}[thm]{Example}
\Crefname{ex}{Example}{Examples}

\theoremstyle{definition}

\Crefname{prob}{Problem}{Problems}

\newtheorem{defn}[thm]{Definition}
\Crefname{defn}{Definition}{Definitions}

\newtheorem*{defn*}{Definition}

\theoremstyle{remark}

\renewenvironment{proof}[1][]{\begin{trivlist}
\item[\hspace{\labelsep}{\bf\noindent Proof#1.\/}] }{\qed\end{trivlist}}

\newcommand{\floor}[1]{
    \left\lfloor #1 \right\rfloor
}

\expandafter\def\expandafter\normalsize\expandafter{%
    \normalsize
    \setlength\abovedisplayskip{8pt}
    \setlength\belowdisplayskip{8pt}
    \setlength\abovedisplayshortskip{4pt}
    \setlength\belowdisplayshortskip{4pt}
}

\usepackage[square,sort,comma,numbers]{natbib}
 \setlist[itemize]{leftmargin=*}

\DeclareFontFamily{OT1}{pzc}{}
\DeclareFontShape{OT1}{pzc}{m}{it}{<-> s * [1.10] pzcmi7t}{}
\DeclareMathAlphabet{\mathpzc}{OT1}{pzc}{m}{it}

\title{\vspace{-0.8cm} 
Intersecting hypergraphs with large cover number}

\author{
Matija Buci\'c\thanks{Department of Mathematics, Princeton University -- Princeton, USA. Email: \href{mailto:mb5225@princeton.edu} {\nolinkurl{mb5225@princeton.edu}}. Research supported in part by an NSF Grant DMS--2349013.}
\and
Vanshika Jain\thanks{Department of Mathematics, Princeton University -- Princeton, USA. Email: \href{mailto:vanshika@princeton.edu} {\nolinkurl{vanshika@princeton.edu}}. Research supported in part by an NSF Graduate Research Fellowship DGE–2039656.}
\and 
Varun Sivashankar\thanks{Department of Mathematics, Princeton University -- Princeton, USA. Email: \href{mailto:varunsiva@princeton.edu}{\nolinkurl{varunsiva@princeton.edu}}}
}
 \date{}

\begin{document}

\maketitle

\vspace{-0.5cm}
\begin{abstract}
In their famous 1974 paper introducing the local lemma, Erd\H{o}s and Lov\'asz posed a question—later referred by Erd\H{o}s as one of his three favorite open problems:
What is the minimum number of edges in an \( r \)-uniform, intersecting hypergraph with cover number \( r \)? This question was solved up to a constant factor in Kahn's remarkable 1994 paper. More recently, motivated by applications to Bollob\'as' ``power of many colours'' problem, Alon, Buci\'c, Christoph, and Krivelevich introduced a natural generalization by imposing a space constraint that limits the hypergraph to use only \(n\) vertices.
In this note we settle this question asymptotically, up to a logarithmic factor in $n/r$ in the exponent, for the entire range.
\end{abstract} 

\section{Introduction} 
The following classical problem was posed by Erd\H{o}s and Lov\'asz in their influential 1974 paper \cite{erdos-lovasz-74}, which introduced the local lemma: 
\begin{qn*}
    What is the minimum number of edges in an \( r \)-uniform, intersecting hypergraph with cover number \( r \)?
\end{qn*}
Here, the \emph{cover number} of a hypergraph is the minimum size of a set of vertices which intersects all the edges and a hypergraph is \emph{intersecting} if any two edges have a non-empty intersection.
Following Erd\H{o}s and Lov\'asz we will denote the answer to this question by \( g(r) \).

In an intersecting $r$-uniform hypergraph, every edge is a cover, so the cover number is at most $r$. Therefore, the question asks for the the minimum number of edges required to achieve the maximum cover number while maintaining the intersecting property. In addition, if we denote by $n$ the number of vertices, ensuring that the cover number is at least \(r\) is equivalent to saying that any vertex subset of size $n-r+1$ must contain an edge (this makes it an instance of the classical hypergraph Tur\'an problem, see \cite{keevash-2011-turan-survey, frankl-rodl-lower-bounds-turan, sidorenko-what-we-know-turan, de-caen-status-turan, sudakov-developments-turan,multi-index-hasing} for the rich history of this type of questions as well as its connections to covering designs and error correcting codes). One can view this as a (very weak) pseudorandom property and indeed random hypergraphs provide very good examples provided one drops the intersecting requirement. On the other hand, being intersecting is a (very weak) structural assumption so the question in a certain sense seeks sparse hypergraphs which exhibit some random-like properties as well as have some structure.

Erd\H{o}s and Lov\'asz 
proved in the original paper that  \[ {8r}/{3} - 3\le g(r) \le 4r^{3/2} \log r, \] 
where the upper bound is conditional on existance of a projective plane of order \(r-1\), and \(r\) being sufficiently large. They conjectured that the true answer is linear, so that \(g(r) = O(r)\). This was one of Erd\H{o}s' three favorite combinatorial problems (see \cite{erdos-favorite-problems}) alongside Erd\H{o}s-Faber-Lov\'asz (which has recently been proved in a remarkable paper \cite{erdos-faber-lovasz-proof} by Kang, Kelly, K{\"u}hn, Methuku, and Osthus) and the Sunflower Conjecture (which while still open has seen some remarkable progress recently \cite{sunflower-lemma-alweiss, sunflower-follow-up-rao, tao-2020-sunflower, bell-chueluecha-lutz-sunflower}).
The Erd\H{o}s-Lov\'asz conjecture was ultimately resolved by Kahn, first up to a lower order term in \cite{kahn-92} and then in full in \cite{kahn-94}.

In a recent paper \cite{power-of-many-colours}, Alon, Buci\'c, Christoph, and Krivelevich, motivated by applications to a certain ``power of many colours'' problem (first introduced by Bollob\'as \cite{bollobas-high-connected}), generalized the question of Erd\H{o}s and Lov\'asz and asked what happens if one is space constrained.
\begin{qn*}
    Let $n \ge 2r-1\ge 3$ be integers. What is the minimum number of edges in an $n$-vertex $r$-uniform intersecting hypergraph with cover number equal to $r$?
\end{qn*}
We denote the answer to this question by $f(n,r)$.
To get an initial feeling for the problem we note that if $n\le 2r-2,$ then any $r-1$ vertices make a cover so such hypergraphs do not exist. If $n=2r-1$, then one is forced to take a complete $r$-uniform hypergraph since the complement of any missing edge would be an $r-1$ cover. This implies that $f(2r-1,r)=\binom{2r-1}{r}$. On the other hand, when $n$ is quadratic, Kahn's construction provides a tight linear bound. Note that an $r$-uniform hypergraph with $O(r)$ edges has at most $O(r^2)$ non-isolated vertices so this follows easily from the result in \cite{kahn-94} even without delving deeper into the intricate construction. On the other hand, a more careful examination of the proof shows the construction, still using projective planes as an ingredient, does actually require a quadratic number of vertices.
This leads to the natural question of how the behaviour of $f(n,r)$ transitions from being exponential in $r$ when $n$ is linear to being linear in $r$ when $n$ is quadratic.

An easy lower bound on $f(n,r)$ can be obtained via a standard double counting argument, due to de Caen, which also gives the essentially best known lower bounds for Turán numbers \cite{de-caen}. Here, on one hand for any of the $\binom{n}{r-1}$ vertex subsets of size $r-1$ there must be an edge disjoint from it (or else this subset provides a cover of size less than $r$). On the other hand, an edge can be disjoint from at most $\binom{n-r}{r-1}$ such subsets showing immediately that 
\begin{equation}\label{eq:de-caen}
    f(n, r) \ge \frac{\binom{n}{r-1}}{\binom{n-r}{r-1}} \ge 2^{\Omega(r^2/n)}.
\end{equation} 

In this paper we give an asymptotic answer to the question of Alon, Buci\'c, Christoph, and Krivelevich by showing this lower bound is essentially tight.

\begin{thm}\label{thm:main}
    For any $3\le  2r-1 \le n \le O(r^2)$, 
    $$f(n,r)= 2^{\tilde{\Theta}(r^2/n)}.$$
\end{thm}

More precisely, we show that there exists an \( n \)-vertex, \( r \)-uniform hypergraph with cover number \( r \) and at most \((n/r)^{O(r^2/n)}\) edges, determining the answer up to a term logarithmic in $n/r$ in the exponent.

For our upper bound on \(f(n, r)\), we make use of a hypergraph product construction that preserves both the intersecting property and criticality in terms of cover number and allows us to interpolate between the complete graph construction and the one found by Kahn \cite{kahn-94} which are optimal at the respective ends of the regime. We describe this product and establish some of its properties relevant for us in \Cref{sec:dproduct}. 
One downside of this construction is that the uniformity of the hypergraphs we construct is always composite and it is surprisingly non-trivial to perform local modifications which preserve both of our desired properties (with one being increasing and one decreasing hypergraph property). Our final ingredient is such a local augmentation lemma which precisely suffices to remove any divisibility requirements on the uniformity. In Section \ref{sec:upper-bound}, we put these arguments together to prove \Cref{thm:main}.

\textbf{Notation.}
Given a hypergraph $H$, we denote by $V(H)$ and $E(H)$, its vertex and edge set, respectively. We will denote by $\tau(H)$ the size of a minimum vertex cover of $H$.
All of our logarithms are in base two unless otherwise specified.

\section{Wreath product of hypergraphs and its properties}\label{sec:dproduct}
We begin by defining a peculiar type of hypergraph product which preserves intersecting and cover criticality properties. We note that the use of this product in combinatorics dates back at least to the paper of Erd\H{o}s and Lov\'asz \cite{erdos-lovasz-74} and independently PhD thesis of Frankl \cite{frankl-thesis} but also appears in group theory \cite{group-theory}. A well-known example in combinatorics is the iterated Fano plane construction. See also \cite{wreath-product}, for a recent application concerned with finding odd sunflowers.

Before giving a formal definition let us describe the product. We start with two hypergraphs $H_1$ and $H_2$ and define their \emph{wreath product} $H_1\rtimes H_2$ to have vertex set consisting of $|V(H_1)|$ many vertex disjoint copies of $V(H_2)$, which we will refer to as \emph{blocks} and identify each block with a vertex of $H_1$. Its edges are all constructed as follows. We pick an edge $e$ of $H_1$ and pick for each $v\in e$ an edge $f_v$ of $H_2$, then we make an edge by taking a disjoint union of the edges $f_v$ where the edge $f_v$ is taken from the block corresponding to $v$. We now give a formal definition.

\begin{defn}
    Given two hypergraphs $H_1$ and $H_2$ we define the \emph{wreath product} hypergraph $H_1 \rtimes H_2$ by 
    \begin{align*}    
    V(H_1 \rtimes H_2)&:=\{(v_1,v_2) \mid v_1 \in V(H_1), v_2\in V(H_2)\} \text{ and}\\
        E(H_1 \rtimes H_2)&:=\big\{ \{(v,u)\mid v \in e, u \in f_v \}\mid e \in E(H_1), \forall v \in e, f_v \in E(H_2)\big\}.
         \end{align*}
\end{defn}

We note that the operation is not commutative, that is $H_1 \rtimes H_2$ is not in general equal to $H_2 \rtimes H_1.$ In the following lemmas, we establish several useful properties of the wreath product of uniform\footnote{We note that we will focus on the uniform case although all the lemmas have, slightly more complicated, general versions as well.} hypergraphs. We begin with some immediate properties to help familiarize the reader with the product.

\begin{lem}\label{lem:properties}
    Let $H_1$ be an $r_1$-uniform and $H_2$ an $r_2$ uniform hypergraph. Then,
    \begin{enumerate}
        \item \label{itm:1} $|V(H_1 \rtimes H_2)|=|V(H_1)|\cdot |V(H_2)|.$
        \item \label{itm:2} $|E(H_1 \rtimes H_2)|=|E(H_1)| \cdot |E(H_2)|^{r_1}$
        \item \label{itm:3} $H_1 \rtimes H_2$ is $r_1r_2$-uniform.
    \end{enumerate}
\end{lem}
\begin{proof}
    \Cref{itm:1} is immediate from the definition. Every edge of $H_1 \rtimes H_2$ is a disjoint union of $r_1$ edges of $H_2$, so has size $r_1r_2$ establishing \Cref{itm:3}. Furthermore, we have $|E(H_1)|$ choices for $e$ and given $e$, $|E(H_2)|$ choices for each $f_v$ for each of the $|e|=r_1$ of $v \in e$. This gives that the number of edges of $H_1 \rtimes H_2$ equals $|E(H_1)| \cdot |E(H_2)|^{r_1}$, establishing \Cref{itm:2}.
\end{proof}

We establish two key properties of the wreath product in the following two lemmas.

\begin{lem}\label{lem:intersecting}
    If $H_1$ and $H_2$ are intersecting hypergraphs, then so is $H_1 \rtimes H_2$.
\end{lem}
\begin{proof}
    Let us take two arbitrary edges $X=\{(v,u)\mid v \in e, u \in f_v \}$ defined by some $e \in E(H_1)$ and $f_v \in E(H_2)$ for all $v \in e$ and $Y=\{(v,u)\mid v \in e', u \in f_v' \}$ defined by some $e' \in E(H_1)$ and $f_v' \in E(H_2)$ for all $v \in e'$. Since $H_1$ is intersecting there is some $v \in e \cap e'$. Since $H_2$ is intersecting there is some $u \in f_v \cap f_v'$. For this $v$ and $u$ we have $(v,u) \in X \cap Y$. Since $X$ and $Y$ were arbitrary this shows $H_1 \rtimes H_2$ is indeed intersecting.
\end{proof}

A key property we will use is that wreath product is multiplicative in terms of the cover number.

\begin{lem}\label{lem:multiplicative-tau}
    Let $H_1$ and $H_2$ be hypergraphs. Then, 
    $$ \tau(H_1 \rtimes H_2) =\tau(H_1)\tau(H_2).$$
\end{lem}
\begin{proof}
    Let $\tau_i=\tau(H_i)$. Suppose $T$ is a cover of $H_1 \rtimes H_2$ and suppose towards a contradiction that $|T|<\tau_1\tau_2$.
    Let $S \subseteq V(H_1)$ consist of vertices $v$ such that $T$ intersects the block corresponding to $v$ in less than $\tau_2$ vertices. Note that $V(H_1) \setminus S$ consists of less than $\tau_1$ vertices, as otherwise, we would have $|T|\ge \tau_1 \cdot \tau_2$. This implies that $V(H_1) \setminus S$ is not a cover for $H_1$ and hence there exists an edge $e \in E(H_1)$ such that $e \subseteq S$. This in turn gives that for any $v \in e$ there exists an $f_v \in E(H_2)$ such that $(v,u) \notin T$ for any $u \in f_v$. With this choice of $e$ and $f_v$'s for $v \in e$ we get an edge $\{(v,u)\mid v \in e, u \in f_v \}$ which is not covered by $T$. This is a contradiction showing $\tau(H_1 \rtimes H_2) \ge \tau_1 \tau_2$.

    To see the upper bound pick a cover $T_i$ of $H_i$ of size $\tau_i$ for $i=1,2$. Now let $T:=\{(v,u) \mid v \in T_1, u \in T_2\}$. Note that $|T|=\tau_1 \tau_2$ and we claim it is a cover of $H_1 \rtimes H_2$.  Consider an edge $X:=\{(v,u)\mid v \in e, u \in f_v \}$ defined by some $e \in E(H_1)$ and $f_v \in E(H_2)$ for all $v \in e$. Note that since $T_1$ is a cover of $H_1$ there exists $v \in e \cap T_1$. Since $T_2$ is a cover of $H_2$ there exists $u \in T_2 \cap f_v.$ Now $(v,u) \in T \cap X$ showing $T$ covers $X$. Since $X$ was arbitrary this shows $T$ is indeed a cover.
\end{proof}

\section{Intersecting hypergraphs with high cover number} \label{sec:upper-bound}
We say that an $r$-uniform hypergraph is \emph{critical} if it is intersecting and has cover number equal to $r$. This definition is motivated by the observation that any intersecting $r$-uniform hypergraph always has cover number at most $r$.

Next, we state formally Kahn's celebrated result, mentioned in the introduction, which provides one of the two hypegraphs used in our product construction—the other being the complete hypegraph. 

\begin{thm}[Kahn, \cite{kahn-94}]\label{thm:kahn}
    There exists $B>0$ such that for any $r \ge 2$ there exists an $r$-uniform critical hypergraph with at most $Br$ edges.
\end{thm}

At our precision level, one could use a simpler, albeit weaker, construction such as a projective plane. We now state and prove the product lemma that lets us interpolate between the two examples. 

\begin{lem}\label{lem:product}
    Given $r_1$- and $r_2$-uniform critical hypergraphs $H_1$ and $H_2$, there exists an $r_1r_2$-uniform critical hypergraph with $|V(H_1)|\cdot |V(H_2)|$ vertices and $|E(H_1)|\cdot |E(H_2)|^{r_1}$ edges.     
\end{lem}
\begin{proof}
    Let our new hypergraph be $H_1 \rtimes H_2$. By \Cref{lem:properties}, its number of vertices, edges, and uniformity are as stated. \Cref{lem:intersecting,lem:multiplicative-tau} ensure that it is critical.
\end{proof}

This result suffices to prove our theorem when the uniformity is the product of small primes since the uniformity of the wreath product graph will be the product of smaller uniformities. However, if our desired uniformity is not of this form—for instance, if it is a prime—\Cref{lem:product} cannot yield a hypergraph with the desired uniformity and number of vertices. To address this, we first construct a hypergraph with uniformity close to the target, then adjust using the following augmentation lemma. This lemma permits a small adjustment to the uniformity at the cost of significantly increasing the number of edges. Since we only apply it a few times, the total edge count remains on the same order as that required by the product construction.

\begin{lem}\label{lem:add-one}
    If there exists an $r$-uniform critical hypergraph $H$, then there exists an $(r+1)$-uniform critical hypergraph with $|V(H)|+r+1$ vertices and at most $(r+1) \cdot |E(H)|+1$ edges.
\end{lem}
\begin{proof}
    We build our new hypergraph $H'$ by taking $V(H')$ to consist of $V(H)$ together with a set $S$ of $r+1$ new vertices. For every edge $e \in E(H)$ and every vertex $v \in S$, we create an edge $e \cup v$. Finally, we add $S$ as an edge of $H'$.

    This construction yields an \((r+1)\)-uniform hypergraph \(H'\) with $|V(H)|+r+1$ vertices and $|E(H)|\cdot |S| +1=|E(H)|\cdot (r+1) +1$ edges. Additionally, $H'$ is intersecting. To see this, observe that any two edges of the form $e \cup v, f \cup u$ (with $e,f \in E(H)$ and $v,u \in S$) intersect because $e$ and $f$ intersect in $H$. The edge \(S\) intersects every other edge since each edge \(e \cup \{v\}\) includes some vertex \(v \in S\).

    Finally, we show the cover number of \(H'\) is \(r+1\). Since \(H'\) is intersecting, the cover number is at most \(r+1\). Suppose, for contradiction, that there is a vertex cover $T$ of $H'$ with size $r$. Because \(S\) is an edge of \(H'\), \(T\) must contain at least one vertex from \(S\); thus, \(|T \setminus S| \leq r - 1\). However, since \(\tau(H) = r\), there exists \(e \in E(H)\) not covered by \(T \setminus S\). Choose \(v \in S \setminus T\) (which exists because \(|S| = r + 1 > r = |T|).\) Then, the edge \(e \cup v\) in \(H'\) is not covered by \(T\), a contradiction. 
\end{proof}

We are now ready to combine our results and prove the following slightly refined version of \Cref{thm:main}.

\begin{thm}
    There exists $C>0$ such that for any $n \ge 2r-1 \ge 3$ there exists an $n$ vertex $r$-uniform critical hypergraph with at most $\max\{(n/r)^{Cr^2/n},Cr\}$ edges.
\end{thm}
\begin{proof}
    Let $B$ be the maximum of the constant provided by \Cref{thm:kahn} and the number $4$, and choose $C = 32B$.
    We may assume that $n \le Br^2$; otherwise, the hypergraph of \Cref{thm:kahn} has at most $B r$ edges and at most $Br^2\le n$ vertices, so by padding with isolated vertices if necessary, we obtain the desired hypergraph. Similarly, we may assume $n \ge 8Br$; if not, the hypergraph $K_{2r-1}^{(r)}$ augmented with isolated vertices $n-2r+1$ produces a hypergraph with at most $4^r\le (n/r)^{4r} \le (n/r)^{32Br^2/n}$ edges. Thus, from now on, we will assume that $8Br \le n \le Br^2$.

    Let $H_1$ be the complete $r_1$-uniform hypergraph on $2r_1-1$ vertices padded with an extra isolated vertex. Note that $H_1$ is critical, has $2r_1$ vertices and has less than $4^{r_1}$ edges.
    Let $H_2$ be an $r_2$-uniform critical hypergraph with at most $Br_2$ edges, provided by \Cref{thm:kahn}. Note that such a hypergraph has at most $Br_2^2$ (non-isolated) vertices, and we may remove the isolated ones.

    We first construct $H=H_1 \rtimes H_2$, which is $r_1 r_2$-uniform, has at most 
    $2Br_1r_2^2$ vertices, and has at most 
    $4^{r_1} \cdot (Br_2)^{r_1}=(4Br_2)^{r_1}$ edges.

    In order to keep our augmentation step within the correct order of magnitude, we split our analysis into two cases. In the first case, if $n<2Br^{3/2}$, we set 
    $$r_2 := \floor{\frac{n}{4Br}}\ge 2, \quad\quad r_1 := \floor{\frac{r}{r_2}} \ge4, \quad\quad t:= r-r_1r_2= r \bmod{r_2}.$$ 
    Note that $0\le t <r_2$.
    
    We now apply \Cref{lem:add-one} a total of $t$ times to increase the uniformity to precisely $r$. By our choice of $r_1,r_2$, the final number of vertices is at most 
    $$2Br_1r_2^2 + t \cdot r \le 2B r r_2+r_2 \cdot r \le 3Brr_2 < n,$$
    which we can pad with isolated vertices if necessary to reach exactly $n$. The number of edges increases to at most $$r^{r_2} \cdot (4Br_2)^{r_1}\le r^{n/(4Br)} \cdot (n/r)^{r/r_2}< (n/r)^{2Br^2/n} \cdot (n/r)^{8Br^2/n}\le (n/r)^{10Br^2/n},$$
    where the first inequality follows from the definitions of $r_1,r_2$ and the second from the assumption $3 \leq n^2/r^2 \le 4B^2 r$ (which implies $\frac{n^2/r^2}{\log (n^2/r^2)}\le \frac{4B^2 r}{\log (4B^2 r)}\le \frac{4B^2 r}{\log r}$ and hence $r^{n/(4Br)} \le (n^2/r^2)^{Br^2/n}$) along with the fact that $r_2 \ge n/(8Br)$.

    If $n \ge 2B r^{3/2}$, we set 
    $$r_1 := \floor{\frac{8Br^2}{n}}\ge 8, \quad\quad r_2 := \floor{\frac{r}{r_1}}\ge 1, \quad\quad t:= r-r_1r_2= r \bmod{r_1}.$$ 
    Note that $0\le t <r_1$.
    
    We now apply \Cref{lem:add-one} a total of $t$ times to increase the uniformity to precisely $r$. The final vertex count is at most
    $$2Br_1r_2^2 + t \cdot r\le 2B r r_2+r_1\cdot r\le \frac{2Br^2}{r_1}+\frac{8Br^3}{n}\le \frac{2Br^2}{4Br^2/n}+\frac{2n}{B} = \frac{n}{2}+\frac{2n}{B} \le n,$$
    where we used $r_1 \ge 4Br^2/n$ in the third inequality and $B \ge 4$ in the final inequality. The number of edges increases to at most $$r^{r_1} \cdot (4Br_2)^{r_1}= (4Brr_2)^{r_1}\le (4Br^2)^{8Br^2/n}\le (n/r)^{32Br^2/n},$$
    using our case assumption (which implies $n/r \ge 2B \sqrt{r}$) in the final inequality.
\end{proof}

\section{Concluding remarks}
In this paper, we answer a vertex-constrained version of the original Erd\H{o}s-Lov\'asz problem—finding the minimum number of edges in an \(r\)-uniform, intersecting hypergraph on \(n\) vertices with cover number \(r\)—posed by Alon, Buci\'c, Christoph, and Krivelevich. We resolve the problem up to a \(\log(n/r)\) factor in the exponent by showing that \(f(n, r) \le  (n/r)^{O(r^2/n)},\) for $n \le O(r^2)$. 

It is natural to wonder whether the simple lower bound \eqref{eq:de-caen} telling us $f(n,r) \geq {\binom{n}{r-1}} \big/{\binom{n-r}{r-1}}$ might be tight. This turns out not to be the case.

\begin{prop}\label{prop:lower-bound}
Let $2r-1\le n < r^2$. Let $H$ be an $n$-vertex, $r$-uniform hypergraph with cover number at least $r$. Then,
\[|E(H)| > \frac{n}{5r} \cdot \frac{\binom{n}{r-1}}{\binom{n-r}{r-1}} \ge \frac{n}{r} \cdot e^{\Omega(r^2/n)}\]
\end{prop}
\begin{proof}
The stated inequality holds by \eqref{eq:de-caen} (which does not make use of the intersecting property) if $n < 5r,$ so we may assume $ n \ge 5r$. Combined with the assumed $r^2$ upper bound we may also assume $r>5$. 

Let $t = \floor{\frac{n}{r}} \le r.$
Pick any set of $r-t$ vertices of $H$. There must exist at least $t$ edges outside it: if there were at most $t-1$ edges outside this set, we can cover them with $t-1$ vertices and obtain a cover of the whole hypergraph of size $r-t+t-1=r-1$.

There are $\binom{n-r}{r-t}$ sets of size $r-t$ living outside any edge, so each edge is counted at most $\binom{n-r}{r-t}$ times. Therefore, it follows that
\begin{align*}
|E(H)| \geq t \cdot \frac{\binom{n}{r-t}}{\binom{n-r}{r-t}} &= t \cdot \frac{\binom{n}{r-1}}{\binom{n-r}{r-1}} \cdot \frac{(n-2r+t)\cdots (n-2r+2)}{(n-r+t)\cdots (n-r+2)}\\ &\geq t \cdot \frac{\binom{n}{r-1}}{\binom{n-r}{r-1}} \cdot \left(1- \frac{r}{n-r+2}\right)^{t-1}\\ &\ge t \cdot \frac{\binom{n}{r-1}}{\binom{n-r}{r-1}} \cdot e^{\frac{-7(t-1)r}{6(n-r+2)}}\ge \frac{t}{e^{7/6}} \cdot \frac{\binom{n}{r-1}}{\binom{n-r}{r-1}}> \frac{n}{5r} \cdot \frac{\binom{n}{r-1}}{\binom{n-r}{r-1}},
\end{align*}
where in the third inequality we used that $1-x \ge e^{-7x/6}$ with $x= \frac{r}{n-r+2} \le \frac{1}{4}$, in the fourth inequality that $t-1 \le \frac{n-r+2}{r},$ and in the final one that $t \ge \frac{5n}{6r}$.
\end{proof}

We made no particular effort to optimize the constant in the above bound. A more careful analysis of the argument shows that for $n \ge 3r$, we get a (slightly) stronger bound compared to \eqref{eq:de-caen}. Moreover, when $n \ge r^2$, the minimum number of edges is clearly $r$ (e.g. a hypergraph consisting of $r$ vertex disjoint edges). Finally, if $H$ is known to be intersecting, the above bound can be improved by roughly a factor of two since in an intersecting hypergraph one can cover any set of up to $2t-2$ edges with $t-1$ vertices. It would be interesting to obtain a more substantial improvement from the intersecting assumption.  

As mentioned in the introduction, the version of the problem without the intersecting assumption leads us to Turán numbers. The Turán number \(T(n, k, r)\) is defined as the maximum number of edges in an \(r\)-uniform hypergraph $H$ on \(n\) vertices that does not contain $K_{k}^{(r)}$. Equivalently, every \(k\)-subset of vertices of $H$ omits some $r$-subset as an edge. This condition is equivalent to the complementary hypergraph $\bar{H}:=\{S \in \binom{[n]}{r}\mid S \notin H\}$ having the $(k,r)$-covering property—that is, every set of $k$ vertices should contain an edge of $\bar{H}$. Let $U(n,k,r)$ denote the minimum number of edges in an $r$-uniform hypergraph on $n$ vertices with the $(k,r)$-covering property. Then, $U(n,k,r)=\binom{n}{r} - T(n, n-r+1, r).$ Since having $(k,r)$-covering property is equivalent to having cover number at least $n-k+1$, determining our function $f(n,r)$ is equivalent to determining the minimum number of edges in an intersecting $r$-uniform hypergraph satisfying the $(n-r+1,r)$-covering property. This implies 
$$f(n,r) \ge U(n,n-r+1,r)=\binom{n}{r} - T(n, n-r+1, r).$$
We note that \Cref{prop:lower-bound} gives a lower bound on $U(n,n-r+1,r)$.

Both the Tur\'an and the covering questions have been extensively studied over the years\footnote{See for example the survey \cite{sidorenko-what-we-know-turan} where all the different perspectives we discuss here are explored.}, although most of the attention was afforded to instances with $r$ and $k$ being fixed. The instance relevant for us is when $k=n-r+1$ and $2r-1 \le n\le O(r^2)$. In this regime, both \(r\) and \(k\) are growing with \(n\) (with $k$ very close to \(n\)). In particular, we are interested in Tur\'an problem for almost spanning hypergraphs. As mentioned above, it is easy to see that $U(n,n-r+1,r)=r$ for $n \geq r^2$.
For our regime, just below this threshold, Sidorenko \cite{sidorenko-thesis} shows that \[T(n, n-r+1, r) = \begin{cases}
    3r - \left\lfloor \frac{2n}{r} \right\rfloor \quad \text{if } r \text{ even} \quad \frac{3r^2}{4} \leq n \leq r^2 \\ 
    3r - \left\lfloor \frac{2(n-r)}{r-1} \right\rfloor \quad \text{if } r \text{ odd} \quad \frac{3r^2 + r}{4} \leq n \leq r^2. 
\end{cases}\]

The best upper bound we are aware of comes from repeatedly taking an edge belonging to the most $n-r+1$ vertex sets not containing any edge already. Since the probability that a random edge belongs to an $n-r+1$ vertex set is $p=\binom{n-r+1}{r}/\binom{n}{r}$, we can always pick out a new edge which belongs to at least this proportion of the $n-r+1$ sets not yet containing one. Hence, after $i$ iterations we are left with at most $(1-p)^i \binom{n}{n-r+1}\le e^{-ip}\binom{n}{r-1}$ sets of $n-r+1$ vertices not containing an edge. This implies we can hit them all with $ p^{-1} \log \binom{n}{r-1}$ edges. Putting together this upper bound with the lower bound from \Cref{prop:lower-bound}, we get
$$ \frac{n}{5r} \le \frac{U(n,n-r+1,r)}{\binom{n}{r-1}/\binom{n-r}{r-1}} < 2r \log \frac{n}r.$$
It would be interesting to close this gap. Returning to our question of separating $f(n,r)$ and $U(n,n-r+1,r)$, with this upper bound in mind, it would suffice to improve \Cref{prop:lower-bound} by making use of the intersecting property to improve the gain factor by only $2r \log \frac nr$.

\textbf{Acknowledgements.}
We want to thank Matthew Kwan and Stefan Glock for useful discussions, and Noga Alon for suggesting a way to prove \Cref{prop:lower-bound}.

\providecommand{\MR}[1]{}
\providecommand{\MRhref}[2]{%
  \href{http://www.ams.org/mathscinet-getitem?mr=#1}{#2}
}

   \bibliographystyle{amsplain_initials_nobysame}
   \bibliography{ref}

\end{document}